\documentclass{amsart}
\usepackage{amssymb}
\usepackage{latexsym,graphicx,pict2e}
\usepackage{hyperref}

\date{\today}

\newcommand{\Z}{{\mathbb Z}}
\newcommand{\R}{{\mathbb R}}

\newcommand{\N}{{\mathbb N}}

\newtheorem{theorem}{Theorem}[section]

\newtheorem{lemma}[theorem]{Lemma}
\newtheorem{prop}[theorem]{Proposition}
\newtheorem{coro}[theorem]{Corollary}
\newtheorem{defi}[theorem]{Definition}
\newtheorem{quest}[theorem]{Question}
\newcommand{\tr}{\mathrm{tr} }

\begin{document}

\title[Schr\"odinger Operators Associated With Aperiodic Subshifts]{Continuum Schr\"odinger Operators Associated With Aperiodic Subshifts}

\author[D.\ Damanik]{David Damanik}

\address{Department of Mathematics, Rice University, Houston, TX~77005, USA}

\email {\href{mailto:damanik@rice.edu}{damanik@rice.edu}}

\urladdr {\href{http://www.ruf.rice.edu/~dtd3}{www.ruf.rice.edu/$\sim$dtd3}}

\thanks{D.\ D.\ was supported in part by a Simons Fellowship and NSF grant DMS--1067988.}

\author[J.\ Fillman]{Jake Fillman}

\address{Department of Mathematics, Rice University, Houston, TX~77005, USA}

\email{\href{mailto:jdf3@rice.edu}{jdf3@rice.edu}}

\author[A.\ Gorodetski]{Anton Gorodetski}

\address{Department of Mathematics, University of California, Irvine, CA~92697, USA}

\email{\href{mailto:asgor@math.uci.edu}{asgor@math.uci.edu}}

\urladdr{\href{http://math.uci.edu/~asgor/}{http://math.uci.edu/~asgor/}}

\thanks{A.\ G.\ was supported in part by NSF grant IIS-1018433.}

\begin{abstract}
We study Schr\"odinger operators on the real line whose potentials are generated by an underlying ergodic subshift over a finite alphabet and a rule that replaces symbols by compactly supported potential pieces. We first develop the standard theory that shows that the spectrum and the spectral type are almost surely constant, and that identifies the almost sure absolutely continuous spectrum with the essential closure of the set of energies with vanishing Lyapunov exponent. Using results of Damanik-Lenz and Klassert-Lenz-Stollmann, we also show that the spectrum is a Cantor set of zero Lebesgue measure if the subhift satisfies the Boshernitzan condition and the potentials are aperiodic and irreducible. We then study the case of the Fibonacci subshift in detail and prove results for the local Hausdorff dimension of the spectrum at a given energy in terms of the value of the associated Fricke-Vogt invariant. These results are elucidated for some simple choices of the local potential pieces, such as piecewise constant ones and local point interactions. In the latter special case, our results explain the occurrence of so-called pseudo bands, which have been pointed out in the physics literature.
\end{abstract}

\maketitle

\section{Introduction}\label{sec:intro}

Operators associated with strictly ergodic subhifts over finite alphabets have been studied in numerous papers since the 1980's. Many of the developments were surveyed in \cite{D}. Until quite recently, most of the effort was devoted to discrete Schr\"odinger operators arising in this way. On the other hand, there have been several works considering continuum Schr\"odinger operators \cite{KLS}, Jacobi matrices \cite{Y}, and CMV matrices \cite{DL2, DMY, MO}. Along the way it has been realized that many tools and results are quite universal in these various instances. On the other hand, when implementing these tools and proving these results, one does have to deal with some model-dependent features. In other words, one may in general be guided by analogies but there does not seem to be an automatic way of carrying over results from one model to another.

The central example in the discrete Schr\"odinger setting is given by the Fibonacci Hamiltonian. Indeed, the foundational papers by Kohmoto, Kadanoff, Tang \cite{KKT} and S\"ut\H{o} \cite{S87, S89} studied this particular model and proved for it the spectral properties, zero-measure spectrum and singular continuous spectral measures, which turned out to be characteristic for the entire class of models. Moreover, the study of this key example is additionally motivated by its relevance to the study of quasicrystals. We refer the reader to the recent survey \cite{DEG}, which further explores this connection. Continuum analogues of the Fibonacci Hamiltonian, especially those with local point interactions leading to the so-called Kronig-Penney model, have been studied in \cite{BJK, Gh, H, KN, KS, TK, WSS}.

The general theory has not yet been worked out to the natural extent possible in the continuum Schr\"odinger setting. Specifically, the paper \cite{KLS} does not address zero-measure Cantor spectrum, and in fact it replaces the use of Kotani theory with Remling's non-ergodic version of it. While this generates a more general result on singular spectrum, it fails to provide the starting point for a proof of zero-measure Cantor spectrum. 

One main motivation for writing this paper is to study the continuum setting with the help of Kotani theory in order to prove zero-measure Cantor spectrum in the same generality as it is known to hold in for discrete Schr\"odinger operators. To the best of our knowledge, this paper gives the first explicit example of a continuum Schr\"odinger operator whose spectrum is a Cantor set of zero Lebesgue measure.\footnote{It appears to be even the first example of a continuum Schr\"odinger operator with potential in $L^2_\mathrm{loc,unif}(\R)$ whose spectrum has zero Lebesgue measure.} In the discrete case, zero-measure spectrum is known to hold also for the critical almost Mathieu operator \cite{AK06, L94}. This, however, is a very unstable phenomenon and quite atypical in the class of operators the almost Mathieu operator belongs to. In particular, it is doubtful whether there is a smooth quasi-periodic continuum potential so that the associated Schr\"odinger operator in $L^2(\R)$ has zero-measure spectrum. In the discrete setting, zero-measure Cantor spectrum is also typical for limit-periodic potentials in the sense of topological genericity; see \cite{A}. No continuum analog is known, but may possibly hold. The zero-measure Cantor spectrum phenomenon is unstable in this scenario as well due to the fact that the periodic potentials are dense in the limit-periodic ones. On the other hand, the zero-measure Cantor spectrum of discrete Schr\"odinger operators associated with suitable strictly ergodic subshifts is both typical and stable within this class of operators. Thus, the latter result is the primary candidate for an attempt to carry over the phenomenon of zero-measure Cantor spectrum to the continuum, and this is precisely what we accomplish here.

The second main motivation for writing this paper is to study the spectrum of the continuum version of the Fibonacci Hamiltonian, which is of interest for several reasons. Since the continuum operator is unbounded, its spectrum is unbounded as well, and hence there are potentially interesting questions one can ask about the high-energy regime. Related to this, the energy-dependence of the local structure of the spectrum may be reduced to the value the so-called Fricke-Vogt invariant takes at the energy in question. This value happens to be constant in the discrete case, while it is not constant (and hence carries more information) in the continuum setting. Secondly, we are able to address a phenomenon that has been pointed out in the physics literature. Namely, the Fibonacci Kronig-Penney model may exhibit so-called pseudo bands, where the spectrum is locally quite thick and in numerical studies (see \cite{BJK, KS}) it is a priori not quite clear if the Cantor structure of the spectrum breaks down at such energies. We show that this is due to the local Hausdorff dimension being one at such points, which in turn comes from zeros of the Fricke-Vogt invariant inside the spectrum. The latter phenomenon is impossible in the discrete case!

\bigskip

The organization of the paper is as follows. In Section~\ref{sec:2} we describe the general setting we will work in. That is, we recall the definition of a subshift over a finite alphabet and describe how we associate continuum Schr\"odinger operators. We also prove the standard results that, given an ergodic measure on the subshift, the spectrum and the spectral type do almost surely not depend on the element of the subshift, and the almost sure spectrum does not contain any isolated points. In Section~\ref{sec:3} we introduce the Lyapunov exponent. To be precise, we introduce both Lyapunov exponents that are natural in this setting -- one associated with the subhift dynamics and one associated with the space variable on the real line -- and prove that they are multiples of one another. Then we prove another standard result, namely that the almost sure absolutely continuous spectrum is given by the essential closure of the set of energies where the Lyapunov exponent vanishes. The work of Klassert, Lenz, and Stollmann \cite{KLS} is then used in Section~\ref{sec:4} to derive the absence of absolutely continuous spectrum whenever the potentials are aperiodic and irreducible. Our main point here is that whenever this result applies, the Lyapunov exponent is almost everywhere positive. This in turn then implies that the spectrum has zero Lebesgue measure, provided that the subshift satisfies the Boshernitzan condition; see Section~\ref{sec:5}. The Cantor structure is a consequence of this since the spectrum is always closed and does not contain isolated points as shown earlier. This completes the general theory for these continuum operators we wish to present. Moving to a special case, we then consider the subshift generated by the Fibonacci substitution in Section~\ref{sec.6}. We establish the usual connection between the spectrum and the dynamics of the trace map, prove a result expressing the local Hausdorff dimension of the spectrum in terms of the value of the Fricke-Vogt invariant at the energy in question, and study the latter quantity for two special cases, piecewise constant potentials and the Kronig-Penney model.

\section{Subshifts Over Finite Alphabets and Associated Continuum Schr\"odinger Operators}\label{sec:2}

In this section we introduce the operators we will study in this work and prove some general results concerning the spectrum and the spectral type of these operators.

First we introduce the notion of the concatenation to make rigorous our notion of ``piecing together functions.''  To that end, assume that for each $n \in \Z$, we have $\ell_n >0$ and a function $f_n:[0,\ell_n) \to \R$.  Assume further that $\sum_{n\geq 0 } \ell_n = \sum_{n < 0} \ell_n = \infty$ (this ensures that the pieced-together function will be defined on all of $\R$).  We define the concatenation of the $f_n$'s by

\[ f(x) = \left\{ \begin{array}{ll} f_n \left( x - \sum_{k=0}^{n-1} \ell_k \right) & n\geq 0 \mbox{ and } x \in \left[ \sum_{k=0}^{n-1} \ell_k, \sum_{k=0}^{n} \ell_k \right), \\ f_n \left( x+\sum_{k=n}^{-1} \ell_k \right) & n < 0 \mbox{ and } x \in \left[- \sum_{k=n}^{-1}\ell_k, -\sum_{k=n+1}^{-1} \ell_k \right). \end{array} \right. \]

We remark that the empty sums which occur for $n=0$ and $n=-1$ are defined to be equal to 0 by convention.  We will notate this concatenation of $f_n$'s by $f=\left( \cdots | f_{-2} | f_{-1}| \, \boxed{f_0} \, | f_1 |f_2| \cdots  \right)$.  Note the use of a box to indicate the position of the origin in the concatenation.  One can also concatenate only finitely many functions - for $m\leq 0 \leq n$, we can produce $\left( f_{m}| \cdots \boxed{f_0} \cdots |f_n \right) $ in the same way, but it will only be defined on $\left[ \sum_{k=m}^{-1} \ell_k, \sum_{k=0}^n \ell_k \right)$.

Let $\mathcal{A}$ be a finite set, called the alphabet. Consider the discrete topology on $\mathcal{A}$ and the product topology on $\mathcal{A}^\Z$. The shift transformation $T : \mathcal{A}^\Z \to \mathcal{A}^\Z$ is given by $(T \omega)(n) = \omega(n+1)$. Clearly, $T$ is a homeomorphism. A subset $\Omega \subseteq \mathcal{A}^\Z$ is called $T$-invariant if $T(\Omega) = \Omega$. Any closed $T$-invariant set $\Omega \subseteq \mathcal{A}^\Z$ is called a subshift.

For every $a \in \mathcal{A}$, we choose a length $\ell_a > 0$ and a real-valued function $f_a \in L^2(0,\ell_a)$. Let a subshift $\Omega \subseteq \mathcal{A}^\Z$ be given. For $\omega \in \Omega$, we define the Schr\"odinger operator $H_\omega = - \frac{d^2}{dx^2} + V_\omega$ in $L^2(\R)$ via $V_{\omega} = \left( \cdots f_{\omega_{-2}} | f_{\omega_{-1}} | \, \boxed{f_{\omega_0}} \, | f_{\omega_1} | f_{\omega_2} \cdots \right)$.  The potentials $V_\omega$ so defined belong to $L^2_\mathrm{loc, unif}(\R)$ and hence define self-adjoint operators in $L^2(\R)$ in a standard way. We will consider the operators $\{H_\omega\}_{\omega \in \Omega}$ as a family and be interested in statements about $H_\omega$ that hold for many, most, or all $\omega \in \Omega$.

Let $S_t$ denote the shift operator $(S_tf)(x) =f(x-t) $ on $L^2(\R)$.  For $\omega \in \Omega$, let $U_{\omega} = S_{\ell_{\omega_0}}$.  One readily sees that $U_{\omega}$ is unitary and a computation reveals that $H_{T\omega} = U_{\omega}^*H_{\omega}U_{\omega}$.  Thus, spectral data will be $T$-invariant, so one should expect the spectrum and spectral parts of $H_{\omega}$ to be nonrandom.  Indeed, one has the following fact:

\begin{prop}
Let $(\Omega,T)$ be a subshift endowed with an ergodic measure $\mu$.  Suppose $\mathcal{A}, \ell_a, f_a, V_{\omega},$ and $H_{\omega}$ are defined as above.  Then there are closed sets $\Sigma,\Sigma_{\textup{ac}},\Sigma_{\textup{sc}},\Sigma_{\textup{pp}} \subset \R$ such that $\sigma_{\bullet}(H_{\omega}) = \Sigma_{\bullet}$ for $\bullet \in \{\, \textup{, ac, sc, pp}\}$ for $\mu$ almost every $\omega\in\Omega$.
\end{prop}

\begin{proof}
The proof is an adaptation of ideas in the literature; see, e.g., \cite{CL, CFKS, KM82}.  We will begin by showing that a suitable family of projections is weakly measurable.

For a fixed interval $I \subset \mathbb{R}$, let $\mathcal{P}_{\omega}(I) = \chi_I(H_{\omega})$.  Fix $f,g \in L^2(\R)$.  We shall show that the map $\omega \mapsto \langle f, \mathcal{P}_{\omega}(I)g \rangle$ is measurable.  First, for each $n \in \Z_+$, let $\Omega^{(n)} = \mathcal{A}^{\{-n,\ldots,n\}}$ and $\pi^{(n)}:\Omega \to \Omega^{(n)}$ the natural restriction map.  For $\omega \in \Omega$, we consider the cut off potential $V_{\omega}^{(n)}$, given by $V_{\omega}^{(n)} = V_{\omega} \chi_{I_{\omega}^{(n)}}$, with $I_{\omega}^{(n)} = (-(\ell_{\omega_{-n}} + \cdots + \ell_{\omega_{-1}}) , \ell_{\omega_0} + \cdots + \ell_{\omega_n})$.  We see that $\omega \mapsto \langle f, e^{itV_{ \omega}^{(n)}} g \rangle$  is measurable for all  $n\in\N , t\in \R$.  But then $\langle f, e^{itV_{ \omega}} g \rangle$ is the limit of measurable functions and hence is itself a measurable function of $\omega$.

By the Trotter product formula, it follows that $e^{i t H_{\omega}}$ is also a weakly measurable funtion of $\omega$.  We can approximate characteristic funtions of intervals with trigonometric polynomials to then see that $\mathcal{P}_{\omega}(I)$ is a weakly measurable function of $\omega$ for all intervals $I$.

Since $\mathcal{P}_{T\omega}(I) = U_{\omega}^*\mathcal{P}_{\omega}(I)U_{\omega}$, it follows that $\omega \mapsto \mbox{tr}(\mathcal{P}_{\omega}(I))$ is measurable and satisfies $\mbox{tr}(\mathcal{P}_{T\omega}(I)) = \mbox{tr}(\mathcal{P}_{\omega}(I))$, so $\mbox{tr}(\mathcal{P}_{\omega}(I))$ is almost surely constant.

For each pair of rational numbers $p<q$, let $d(p,q)$ denote the almost sure value of $\mbox{tr}(\mathcal{P}_{\omega}((p,q)))$ and $\Omega_{p,q}$ the (full measure) set containing those $\omega\in\Omega$ for which $\mbox{tr}(\mathcal{P}_{\omega}((p,q))) = d(p,q)$.  Since the set of rational pairs is countable, it follows that $\Omega_0 = \bigcap_{p<q}\Omega_{p,q}$ is a full measure subset of $\Omega$.

Suppose $E \notin \sigma(H_{\omega_1})$ for some $\omega_1 \in \Omega_0$.  Choose rational $p<q$ so that $E \in (p,q) \subset \rho(H_{\omega_1})$.  Then $d(p,q) = \mbox{tr}(\mathcal{P}_{\omega_1}((p,q))) = 0$.  For any other $\omega_2 \in \Omega_0$, one has $\mbox{tr}(\mathcal{P}_{\omega_2}((p,q))) = d(p,q) = 0$, so $E \notin \sigma(H_{\omega_2})$.  By symmetry, there is some nonrandom closed set $\Sigma \subset \R$ with $\sigma(H_{\omega}) = \Sigma$ for all $\omega \in \Omega_0$.

For the almost sure constancy of spectral parts, let $\mathcal{P}_\omega^{\textup{ac}}$ denote orthogonal projection onto the absolutely continuous subspace of $H_{\omega}$ and define $\mathcal{P}_\omega^{\textup{sc}}$ and $\mathcal{P}_\omega^{\textup{pp}}$ analogously.  Next, let $\mathcal{P}_{\omega}^{\bullet}(I) = \mathcal{P}_{\omega}(I) \mathcal{P}_{\omega}^{\bullet}$, for $\bullet \in \{\textup{ac, sc, pp}\}$.  Running the argument above with $\mathcal{P}_{\omega}(I)$ replaced by $\mathcal{P}_{\omega}^{\bullet}(I)$ (and a measurability proof as in \cite{CL, KM82}) gives us a full measure set $\Omega' \subseteq \Omega$ and nonrandom sets $\Sigma_{\textup{ac}},\Sigma_{\textup{sc}},\Sigma_{\textup{pp}}$ so that $\sigma_{\bullet}(H_{\omega})= \Sigma_{\bullet}$ for all $\omega \in \Omega'$ and $\bullet \in \{\textup{ac, sc, pp}\}$.
\end{proof}

In the usual setting of ergodic dynamically defined potentials, the spectrum lacks isolated points.  This fact carries over to the present setting.

\begin{lemma}\label{isolated}
Let $(\Omega,T)$ be a subshift endowed with an ergodic measure $\mu$.  Suppose $\mathcal{A}, \ell_a, f_a, V_{\omega}, H_{\omega},$ and $\Sigma$ are defined as above.  Then $\Sigma$ has no isolated points.
\end{lemma}

\begin{proof}
It is convenient for us to work with measures on the space of potentials, so we let $\nu$ be the push-forward of $\mu$ under the map $\Phi:\omega \mapsto V_{\omega}$.  We also will find it convenient to work with $\R$-invariant families of potentials, so, to that end, let $\tilde{\Omega} = \Omega \times [0,1)$.  For each $(\omega,t)\in\tilde{\Omega}$, we define $V_{\omega,t}(x) = V_{\omega}(x+t\ell_{\omega_0})$.  Let $\tilde{\mu}$ denote the product of $\mu$ and Lebesgue measure, and $\tilde{\nu}$ be the pushforward of $\tilde{\mu}$ under the map $\Psi:(\omega,t)\mapsto V_{\omega,t} $.  Denote $X=\{V_{\omega,t} : \omega\in\Omega , t\in [0,1)\}$

As in the proof of the previous proposition, we fix an interval $I$ and take $\mathcal{P}_{\omega,t} = \mathcal{P}_{\omega,t}(I) = \chi_I(H_{\omega,t})$ with $H_{\omega,t} = -\Delta + V_{\omega,t}$.  An argument similar to the one in the above proposition establishes that $\textup{tr}(\mathcal{P}_{\omega,t}) = \textup{tr}(\mathcal{P}_{\omega})$.

Let $\{g_n : n\in\Z\}$ be an orthonormal basis for $L^2(0,1)$ and set $g_{n,m} = S_m g_n$.  It follows that $\{g_{n.m} :n,m\in\Z\}$ is an orthonormal basis of $L^2(\R)$.

Let $d$ denote the $\mu$ almost sure value of $\textup{tr}(\mathcal{P}_{\omega})$ and use $\mathbb{E}$ to denote integration over a given probability space.  One can then observe

\begin{align*}
d & = \mathbb{E}_{\Omega}(\textup{tr}(\mathcal{P}_{\omega})) \\
  & = \mathbb{E}_{\tilde\Omega} (\textup{tr}(\mathcal{P}_{\omega,t})) \\
  & = \mathbb{E}_{\tilde\Omega} \left( \sum_{n,m} \langle g_{n,m}, \mathcal{P}_{\omega,t} g_{n,m} \rangle \right)\\
  & = \sum_{n,m}  \mathbb{E}_{\tilde\Omega} \left(  \langle g_{n,m}, \mathcal{P}_{\omega,t} g_{n,m} \rangle \right) \\
  & = \sum_{n,m}  \mathbb{E}_{\tilde\Omega} \left(  \langle g_n, S_m^*\mathcal{P}_{\omega,t} S_m g_n \rangle \right) \\
  & = \sum_{n,m} \int_{X} \! \langle g_n, S_m^*\chi_I(-\Delta+V) S_m g_n \rangle \, \mathrm{d}\tilde{\nu}(V)  \\
  & = \sum_{n,m} \int_{X} \! \langle g_n, \chi_I(-\Delta+V)  g_n \rangle \, \mathrm{d}\tilde{\nu}(V).
\end{align*}

The last equality holds because $\tilde{\nu}$ is invariant under the transformations $(S_r)_{r\in\R}$.  In particular, $d$ is either 0 or infinity according to whether or not $\int_{X} \! \langle g_n, \chi_I(-\Delta+V)  g_n \rangle \, \mathrm{d}\tilde{\nu}(V)$ is 0 for all $n$ or positive for some $n$, respectively.

It follows that $\Sigma$ cannot have an isolated point.  To see this, observe that if $E\in\Sigma$ is isolated, one can choose an interval $I$ with $I \cap  \Sigma = \{E\}$.  But since solution spaces are two-dimensional, this would require $1 \leq \textup{tr}(\mathcal{P}_{\omega}(I)) \leq 2$ (for $\mu$ almost every $\omega$), a contradiction. Hence, $\Sigma$ lacks isolated points.
\end{proof}

\section{The Lyapunov Exponent}\label{sec:3}

We can associate a discrete $\mathrm{SL}(2,\R)$ cocycle to our model as follows.  For a given energy $E$, $\omega \in \Omega$, and $x_0,x \in \R$, there is a unique unimodular matrix $A_{E,\omega}(x,x_0)$ so that $A_{E,\omega} (x,x_0) \left( \begin{array}{c} y'(x_0) \\ y(x_0) \end{array} \right) = \left( \begin{array}{c} y'(x) \\ y(x) \end{array} \right)$ for any solution $y$ of $-y'' + V_{\omega}y =  Ey$.  We note that $A$ obeys the differential equation
\begin{align*}
\frac{\partial A_{E,\omega}}{\partial x}(x,x_0) & = B_{E,\omega}(x) A_{E,\omega}(x,x_0) \\
A(x_0,x_0) & = \left( \begin{array}{cc} 1 & 0 \\ 0 & 1 \end{array} \right),
\end{align*}
where $B_{E,\omega}(x)= \left( \begin{array}{cc} 0 & V_{\omega}(x)-E \\ 1 & 0 \end{array} \right)$.  We may note that this yields an elementary bound on the growth of these transfer matrices:
$$
\|A_{E,\omega}(x,x_0)\| \leq \exp \left( \int_{x_0}^x \| B_{E,\omega}(x) \| \, \mathrm{d}x \right) = \exp \left( \int_{x_0}^x \max \left(1,|V_{\omega}(x)-E| \right) \, \mathrm{d}x \right).
$$

Now, for each $n \in \Z_+$, we let $A_{E,\omega}^n$ denote the transfer matrix over the subword $\omega_0 \ldots \omega_{n-1}$ of $\omega$ of length $n$, i.e., $A_{E,\omega}^n = A_{E,\omega}(\ell_{ \omega_0 }+ \cdots + \ell_{\omega_{n-1}},0)$.  One readily sees that the cocycle condition holds:
$$
A_{E,\omega}^{n+m} = A_{E,T^n\omega}^m A_{E,\omega}^n.
$$
In particular, we may invoke Kingman's Subadditive Ergodic Theorem to see the following.

\begin{prop}
For each $E \in \mathbb{R}$, there is a constant $L_{\textup{disc}}(E)$ such that $\lim_{n \to \infty} \frac{1}{n} \log \| A_{E,\omega}^n \| = L_{\textup{disc}}(E)$ for $\mu$ almost every $\omega \in \Omega$.
\end{prop}

This discrete cocycle's asymptotic behavior is related to the continuum cocycle's behavior in a natural way.

\begin{prop}
Enumerate $\mathcal{A} = \{a_1,\ldots,a_k\}$ and define $\Omega_j = \{ \omega : \omega_0 = a_j \}$.  Let $s = \ell_{a_1} \mu(\Omega_1) + \cdots + \ell_{a_k} \mu(\Omega_k)$, the weighted average of the lengths.  For almost every $\omega$, the continuum Lyapunov exponent $L(E) = \lim_{x\to\infty} \frac{1}{x} \log\|A_{E,\omega}(x,0)\|$ exists and satisfies $L(E) = \frac{L_{\textup{disc}}(E)}{s}$.
\end{prop}

\begin{proof}
Consider $f:\Omega \to \mathbb{R}$ given by $f(\omega) = \ell_{\omega_0}$ Then,
$$
\frac{1}{N} \sum_{n = 0}^{N-1} f(T^n\omega) = \frac{\ell_{\omega_0} + \cdots + \ell_{\omega_{N-1}}}{N}.
$$
Hence, by Birkhoff's ergodic theorem, we have
$$
\lim_{n \to \infty} \frac{\ell_{\omega_0} + \cdots + \ell_{\omega_{n-1}}}{n} = \mathbb{E}(f) = s
$$
for $\mu$ almost every $\omega \in \Omega$.

Since $\lim_{n \to \infty} \frac{1}{n} \log \left\| A_{E,\omega}^n \right\|$ and $\lim_{n \to \infty} \frac{\ell_{\omega_0} + \cdots + \ell_{\omega_{n-1}}}{n}$ exist for $\mu$ almost every $\omega$, we have that
$$
\lim_{n \to \infty} \frac{1}{\ell_{\omega_0} + \cdots + \ell_{\omega_{n-1}}} \log \left\| A_{E,\omega}^n \right\|  = \frac{L_{\textup{disc}}(E)}{s}.
$$
Thus, we see that the limit defining the continuum Lyapunov exponent exists and equals $s^{-1} L_{\textup{disc}}(E)$ along a subsequence.

Next, we want to estimate the error in approximating the matrix $A_{E,\omega}(x,0)$ with a nearby point $A_{E,\omega}(x_n,0)$, where $x_n= \ell_{\omega_0} + \cdots + \ell_{\omega_{n-1}}$.  To that end, assume given $x > 0$ and choose $n$ with $x_n \leq x < x_{n+1}$. We observe the following estimates:
$$
\frac{1}{x} \log\|A_{E,\omega}(x,0)\| - \frac{1}{x_n} \log \|A_{E,\omega}(x_n,0)\| \leq \frac{1}{x_n}\log\|A_{E,\omega}(x,x_n) \|
$$
and
\begin{align*}
\frac{1}{x_n} \log\|A_{E,\omega}(x_n,0)\| & - \frac{1}{x} \log \|A_{E,\omega}(x,0)\| \\
& \leq \left(\frac{1}{x_n} - \frac{1}{x} \right) \log \|A_{E,\omega}(x,0) \| + \frac{1}{x_n} \log\|A_{E,\omega} (x_n,x)\|.
\end{align*}

As $x$ goes to infinity, these estimates go to zero.  In particular, there is a full measure set of $\omega$ for which $L(E)$ exists and is equal to $L_{\textup{disc}}(E)/s$
\end{proof}

With this relationship established, we may consider $\mathcal{Z} = \{E : L(E) = 0\} = \{E:L_{\textup{disc}}  = 0 \}$.  The celebrated theorem of Ishii-Pastur-Kotani extends naturally to our setting:

\begin{prop}
$\Sigma_{\textup{ac}} = \overline{\mathcal{Z}}^{\textup{ess}}$
\end{prop}

\begin{proof}
Let $\nu,\tilde{\Omega},\tilde{\mu}, X, \tilde{\nu},$ and $V_{\omega,t}$ be defined as in the proof of Lemma~\ref{isolated}. We will show that $\tilde\nu$ is ergodic with respect to the family of maps $\{S_r:r\in\R\}$.

Invariance is clear. To prove ergodicity, we follow Kirsch's argument from \cite{K}. Assume that $A$ is invariant under $(S_r)_{r\in\R}$.  For each $t$, we may consider the section $A_t = \{ \omega \in \Omega : V_{\omega,t}\in A \}$.  Notice that
$$
T^{-1}[A_t] = \left\{\omega\in\Omega : V_{T\omega,t} \in A \right\} =  \left\{\omega \in \Omega : V_{\omega,t} \in S_{\ell_{\omega_0}}[A]=A \right\} = A_t
$$
for each $t$, so that $\mu(A_t) = 0$ or 1 for each $t$ by ergodicity of $(\Omega,\mu,T)$.  Moreover, we have
$$
A_t = \{ \omega \in\Omega: V_{\omega,t} \in A \} = \left\{ \omega \in\Omega: V_{\omega,0} \in S_{t\ell_{\omega_0}}[A]=A \right\}  = A_0
$$
for all $t$.  We then observe that $\displaystyle \tilde{\nu} (A) = \tilde{\mu} (\Psi^{-1}[A]) = \int_0^1 \! \mu (A_t) \, \mathrm{d}t = \mu(A_0) $ and hence $\tilde{\nu} (A)$ is 0 or 1. It follows that $\tilde{\nu}$ is ergodic.

We can define the Lyapunov exponent on $\tilde{\Omega}$ by $L(E,\omega,t) = \lim_{x\to\infty}\frac{1}{x} \log\|A_{E,(\omega,t)} (x,0)\| $, where $A_{E,(\omega,t)}(x,0)$ is the natural $\mathrm{SL}(2, \R)$ cocycle associated to $H_{\omega,t} = -\Delta + V_{\omega,t}$.  Standard arguments show that $L(E,\omega,t)$ exists for $\tilde{\mu}$ almost every $(\omega,t)$ and is $\tilde{\mu}$ almost surely equal to a constant $\tilde{L}(E)$.  One sees that $A_{E,(\omega,t)}(x,0) = A_{E,\omega}(x+t,t)$.

Let $S_1$ denote the full measure set of $V_{\omega,t}$ for which
$$
\lim_{x\to \infty}\frac{1}{x}\log\|A_{E,{\omega,t}}(x,0)\| = \tilde{L}(E).
$$

Next, let $S_2$ consist of those $V_{\omega}$ such that $V_{\omega,t}\in S_1$ for some $t$ (this set has full $\nu$ measure).  Next, let $S_3$ consist of those $V_{\omega}$ for which
$$
\lim_{x\to \infty}\frac{1}{x}\log\|A_{E,\omega}(x,0)\| = L(E).
$$
Since $S_2$ and $S_3$ have full $\nu$ measure, we may choose $V_{\omega }\in S_2\cap S_3$ and then $t$ such that $V_{\omega,t}\in S_1$. But then, we have
$$
L(E) = \lim_{x\to \infty}\frac{1}{x}\log\|A_{E,\omega}(x,0)\|  = \lim_{x\to \infty}\frac{1}{x}\log\|A_{E,(\omega,t)}(x,0)\| = \tilde{L}(E).
$$

By standard arguments, there is some nonrandom set $\tilde{\Sigma}_{\textup{ac}}$ such that  $\tilde{\Sigma}_{\textup{ac}} = \sigma_{\textup{ac}}(H_{\omega,t})$ for $\tilde{\mu}$ almost every $(\omega,t)$.  Since $H_{\omega,t}$ is unitarily equivalent to $H_{\omega}$, it follows that $\tilde{\Sigma}_{\textup{ac}} = \Sigma_{\textup{ac}}$.  Moreover, the Ishii-Pastur-Kotani theorem \cite{CL, CFKS} implies that $\tilde{\Sigma}_{\textup{ac}} = \overline{\tilde{L}^{-1}(0)}^{\textup{ess}} = \overline{\mathcal{Z}}^{\textup{ess}}$.  Thus, $\Sigma_{\textup{ac}} = \overline{\mathcal{Z}}^{\textup{ess}}$.
\end{proof}

\section{Absence of Absolutely Continuous Spectrum}\label{sec:4}

Next, we would like to prove that the absolutely continuous spectrum $\Sigma_{\textup{ac}} $ is empty for aperiodic models.  However, there is a slight complication -  even if $(\Omega,T)$ is aperiodic, it does not necessarily follow that the potentials are aperiodic.  For example, we could have $f_{a}(x) = f_b(x) = \sin(x)$, $\ell_a = 2\pi$, $\ell_b = 4\pi$.  In this case, an aperiodic word containing only $a$'s and $b$'s will give rise to the periodic potential $V(x) = \sin(x)$.  To remove degenerate potentials such as this, we must impose the following addditional assumptions:

\begin{itemize}

\item Aperiodicity: The subshift $\Omega$ and the collection $\{f_a : a \in \mathcal{A}\}$ are such that the potentials $V_{\omega}$ are not periodic.

\item Irreducibility: There is some $\ell > 0$ such that the following holds.  Suppose two ``pieces'' of potential satisfy $\chi_{[0,\ell)} \left( \boxed{f_{a_1}} | f_{a_2} | \cdots | f_{a_k} \right) = \chi_{[0,\ell)} \left( \boxed{f_{b_1}} | f_{b_2} | \cdots | f_{b_n} \right)$ (Note that this requires $\ell < \ell_{a_1} + \cdots \ell_{a_k},\ell_{b_1} + \cdots \ell_{b_n}$).  Then $a_1 = b_1$.  In some sense, the potential has a unique decomposition into a concatenation of the finite pieces $\{f_a : a \in\mathcal{A} \}$.

\end{itemize}

Notice that aperiodicity need not imply irreducibility.  To see this, consider $f_a(x) = f_b(x) = \sin(x)$,  $f_c(x) = \sin(x-\pi)$, $\ell_a = 2\pi$, $\ell_b = \ell_c = \pi$.

Under these conditions, the potentials satisfy the simple finite decomposition property, as described in \cite{KLS}. This enables us to prove the following.

\begin{prop}
$\Sigma_{\textup{ac}} = \emptyset$
\end{prop}

\begin{proof}
By \cite[Theorem~4.1]{KLS}, if the restriction of $H_{\omega}$ to either half line has nonempty absolutely continuous spectrum, then said half-line restriction must have a potential which is eventually periodic.  As the potentials under consideration are not eventually periodic, it follows that both half-line restrictions have empty absolutely continuous spectrum.  By general considerations, the absolutely continuous spectrum of the whole-line operator $H_{\omega}$ is equal to the union of the absolutely continuous spectra of its half line restrictions.  Hence, $\Sigma_{\textup{ac}}=\emptyset$, as desired.
\end{proof}

By general measure-theoretic considerations, this immediately implies

\begin{coro}\label{no.ac}
The Lebesgue measure of $\mathcal{Z}$ is zero.
\end{coro}

\begin{proof}
Suppose  $\overline{S}^{\textup{ess}} = \emptyset$.  Then, given a compact interval $I$, since the essential closure of $S$ is empty, $I$ is covered by finitely many open intervals $J_1,\ldots,J_m$ such that $|J_k \cap S| = 0$ for $1 \leq k \leq m$.  But this implies $|S \cap I|=0$ for all compact intervals $I$ and hence $|S| = 0$.  In particular, the essential closure of $\mathcal{Z}$ is empty, so $\mathcal{Z}$ must have zero Lebesgue measure.
\end{proof}

\section{Absence of Non-Uniform Hyperbolicity and Cantor Spectrum}\label{sec:5}

In this section we prove our main result in the general setting, namely a sufficient condition for zero-measure Cantor spectrum.

\begin{defi}
Let $A_{E,\omega}(x,0)$ be as above.  We say that $E \in \mathcal{UH}$ if there are constants $C,\gamma>0$ such that $\|A_{E,\omega}(x,0)\| \geq Ce^{\gamma|x|}$ for every $x \in \R$.  That is, the transfer matrices grow exponentially in $x$ and uniformly on $\Omega$.
\end{defi}

Uniform hyperbolicity has several equivalent formulations; compare, for example, \cite{Y04}. One of them is used in the proof of the following proposition.

\begin{prop}\label{p.sincompuh}
$\Sigma \subset \R \backslash \mathcal{UH}$
\end{prop}

\begin{proof}
Suppose $E \in \mathcal{UH}$. Then, for every $\omega \in \Omega$, there exist linearly independent solutions $u_{\pm}$ of the equation $-u'' + V_\omega u = Eu$ with $L^2$ decay at $\pm \infty$.  A standard argument then shows that
\[ G_\omega(E,x,y) = \frac{u_{-}(\min(x,y)) u_{+}(\max(x,y))}{u_{-}'(x)u_{+}(x) - u_{-}(x)u_{+}'(x)} \]
defines the integral kernel of $(H_\omega - E)^{-1}$.  In particular, $E \notin \Sigma$.
\end{proof}

Recall the following definition of condition (B) for a minimal subshift, introduced by Boshernitzan in \cite{B1}.

\begin{defi}
Let $\Omega$ be a minimal subshift. It satisfies the Boshernitzan condition {\rm (B)} if there exists a $T$-ergodic measure $\mu$ such that
$$
\limsup_{n \to \infty} \left( \min_{w \in \mathcal{W}_\Omega(n)} n \cdot \mu \left( [w] \right) \right) > 0.
$$
Here, $[w]$ denotes the cylinder set determined by a finite word $w$ and $\mathcal{W}_\Omega(n)$ denotes the set of words of length $n$ that occur in elements of $\Omega$.
\end{defi}

We remark that condition (B) implies unique ergodicity of $(\Omega,T)$ \cite{B3}. All subshifts generated by primitive substitutions (this includes in particular the case of the Fibonacci substitution considered later in the paper) and, more generally, all linearly recurrent subshifts satisfy (B). See \cite{DL06b} for many more examples of subshifts satisfying (B).

\begin{prop}\label{p.zeqcompuh}
Suppose $(\Omega,T)$ is a minimal subshift which satisfies {\rm (B)}.  Then $\mathcal{Z} = \R \backslash \mathcal{UH}$.
\end{prop}

\begin{proof}
By \cite[Theorem~1]{DL}, it follows that $\frac{1}{n} \log \|A_{E,\omega}^n\|$ converges to $L_{\textup{disc}}(E)$ uniformly on $\Omega$.  From this and the estimates in Section~3, it follows that $\frac{1}{x} \log \|A_{E,\omega}(x,0) \|$ converges uniformly on $\Omega$ to $L(E)$.  But then for $E \in \R$, we have $E \notin \mathcal{Z} \iff L(E)>0 \iff E \in \mathcal{UH}$.
\end{proof}

We are now in a position to prove zero-measure Cantor spectrum. Here, we deviate from standard conventions slightly and call a subset of $\R$ a \emph{Cantor set} if it is closed, does not contain any isolated points, and has empty interior.

\begin{coro}\label{c.zmcs}
Suppose $(\Omega,T)$ is a minimal subshift which satisfies {\rm (B)} and such that the associated potentials $V_{\omega}$ are aperiodic and irreducible in the sense described in Section~\ref{sec:4}.  Then $\Sigma$ is a Cantor set of Lebesgue measure zero.
\end{coro}

\begin{proof}
As it is the spectrum of an operator, $\Sigma$ is closed. Lemma \ref{isolated} tells us that $\Sigma$ lacks isolated points.  By Propositions~\ref{p.sincompuh} and \ref{p.zeqcompuh}, we have $\Sigma \subseteq \mathcal{Z}$.  The opposite inclusion, $\Sigma \supseteq \mathcal{Z}$, follows from general principles; compare, for example, the proof of \cite[Theorem~2.9]{CFKS}. Thus, $\Sigma = \mathcal{Z}$. Combining this with Corollary~\ref{no.ac}, which says that $|\mathcal{Z}|=0$, it follows that $\Sigma$ is nowhere dense and of zero Lebesgue measure.  Hence, $\Sigma$ is a Cantor set of zero Lebesgue measure.
\end{proof}

As there are many aperiodic subshifts that satisfy (B) \cite{DL06b} and most choices of local potential pieces yield aperiodic and irreducible potentials $V_\omega$, this result provides a large family of continuum Schr\"odinger operators whose spectrum is a Cantor set of zero Lebesgue measure. As pointed out in the introduction, to the best of our knowledge, no previous examples with this spectral property were known.

\section{The Case of the Fibonacci Subshift}\label{sec.6}

In this section we study a special case, namely the subshift generated by the Fibonacci substitution. The alphabet is given by $\mathcal{A} = \{ a,b \}$. The Fibonacci substitution is the map $S(a) = ab$, $S(b) = a$. This map extends by concatenation to finite words over $\mathcal{A}$, as well as one-sided infinite words over $\mathcal{A}$. There is a unique one-sided infinite word $u$ over $\mathcal{A}$ that obeys $u = S(u)$. It is obtained as the limit (in the obvious sense) of the sequence of finite words $\{ S^n(a) \}_{n \in \Z_+}$. The Fibonacci subshift $\Omega \subseteq \mathcal{A}^\Z$ is defined by $\Omega = \{ \omega \in \mathcal{A}^\Z : \text{every subword of $\omega$ is a subword of } u \}$. It is easy to see (cf.~\cite{DL06b}) that it satisfies (B). Given a choice of $\ell_a, \ell_b > 0$ and real-valued functions $f_a \in L^2(0,\ell_a)$, $f_b \in L^2(0,\ell_b)$, we consider as above the Schr\"odinger operators $\{H_\omega\}_{\omega \in \Omega}$. As usual, we assume that aperiodicity and irreducibility hold unless otherwise stated.

We may apply Corollary~\ref{c.zmcs} and deduce the following:

\begin{theorem}\label{t.fibcantor}
There is a Cantor set $\Sigma \subseteq \R$ of zero Lebesgue measure such that $\sigma(H_\omega) = \Sigma$ for every $\omega \in \Omega$.
\end{theorem}

Our goal is to go beyond this statement and to study fractal properties of $\Sigma$, that is, quantities such as the local Hausdorff dimension at a point of $\Sigma$. The properties of the spectrum are reflected by the dynamical properties of the so-called Fibonacci trace map. Namely, the set of all energies corresponds to a curve of initial conditions (which is model dependent), and a given energy belongs to the spectrum if and only if the positive semiorbit of the corresponding initial condition under the action of the trace map is bounded. Two specific cases (piecewise constant potential and Kronig-Penney model) have been considered in the physics literature most often (mostly via numerical experiments); see \cite{BJK, Gh, H, KN, KS, TK, WSS}. Here we provide rigorous results confirming and explaining the previous numerical observations for each of these models. It is interesting to notice that there is an essential difference in the spectral properties between these models. Namely, in the Kronig-Penney model there are values of the energy that belong to the spectrum regardless of the value of coupling constant, and where the local Hausdorff dimension of the spectrum is equal to one (therefore the so called ``pseudo bands'' are formed, see also \cite{BJK}). At the same time, these pseudo bands do not appear in the piecewise constant case, where the Hausdorff dimension of the spectrum in any compact domain tends to zero as the coupling constant tends to infinity.

\subsection{Trace Map, Fricke-Vogt Invariant, and Local Hausdorff Dimension of the Spectrum}

As is well known in the discrete case, the spectrum (and many spectral properties) of the Fibonacci model can be described in terms of the trace map. Let us make this connection explicit. Consider the solutions of the differential equation $-u''(x) + f_a(x) u(x) = E u(x)$ for real energy $E$. Denote the solution obeying $u(0) = 0$, $u'(0) = 1$ (resp., $u(0) = 1$, $u'(0) = 0$) by $u_{a,D}(\cdot,E)$ (resp., $u_{a,N}(\cdot,E)$). Similarly, by replacing $f_a$ with $f_b$, we define $u_{b,D}(\cdot,E)$ and $u_{b,N}(\cdot,E)$. Then, we set
\begin{align*}
M(a,E) & = \begin{pmatrix} u_{a,N}(\ell_a,E) & u_{a,D}(\ell_a,E) \\ u_{a,N}'(\ell_a,E) & u_{a,D}'(\ell_a,E) \end{pmatrix}, \\
M(b,E) & = \begin{pmatrix} u_{b,N}(\ell_b,E) & u_{b,D}(\ell_b,E) \\ u_{b,N}'(\ell_b,E) & u_{b,D}'(\ell_b,E) \end{pmatrix},
\end{align*}
and
\begin{align}
x_{-1}(E) & = \frac12 \tr \left( M(b,E) \right), \label{e.x-1} \\
x_0(E) & = \frac12 \tr \left( M(a,E) \right), \label{e.x0} \\
x_1(E) & = \frac12 \tr \left( M(b,E) M(a,E) \right) . \label{e.x1}
\end{align}
The map $E \mapsto (x_2(E), x_1(E), x_0(E))$ will be called the curve of initial conditions. The trace map $T : \R^3 \to \R^3$ is given by
$$
T(x,y,z) = (2xy-z,x,y).
$$

Define
$$
I(x,y,z) = x^2 + y^2 + z^2 - 2xyz - 1.
$$
Again, as is well known, the trace map $T$ preserves $I$ and hence its level surfaces
$$
S_I = \{ (x,y,z) \in \R^3 : I(x,y,z) = I \}.
$$
In particular, all the points $T^n(x_1(E), x_0(E), x_{-1}(E))$ lie on the surface $S_{I(E)}$, where (with some abuse of notation) we set
$$
I(E) = I(x_1(E), x_0(E), x_{-1}(E)).
$$

The surfaces $S_I$ undergo a transition at $I = 0$. For negative values of $I$, $S_I$ has one bounded connected component and four unbounded connected components. For $I = 0$, the four ``outside'' components attach to the inner part in four conic singularities. These singularities resolve as $I$ becomes positive, and in this case $S_I$ is connected and smooth. We show plots of $S_I$ for $I < 0$, $I = 0$, and $I > 0$ in Figure~1.

\bigskip

\begin{figure}\label{f.levsurf}
\begin{center}
   \includegraphics[scale=0.3]{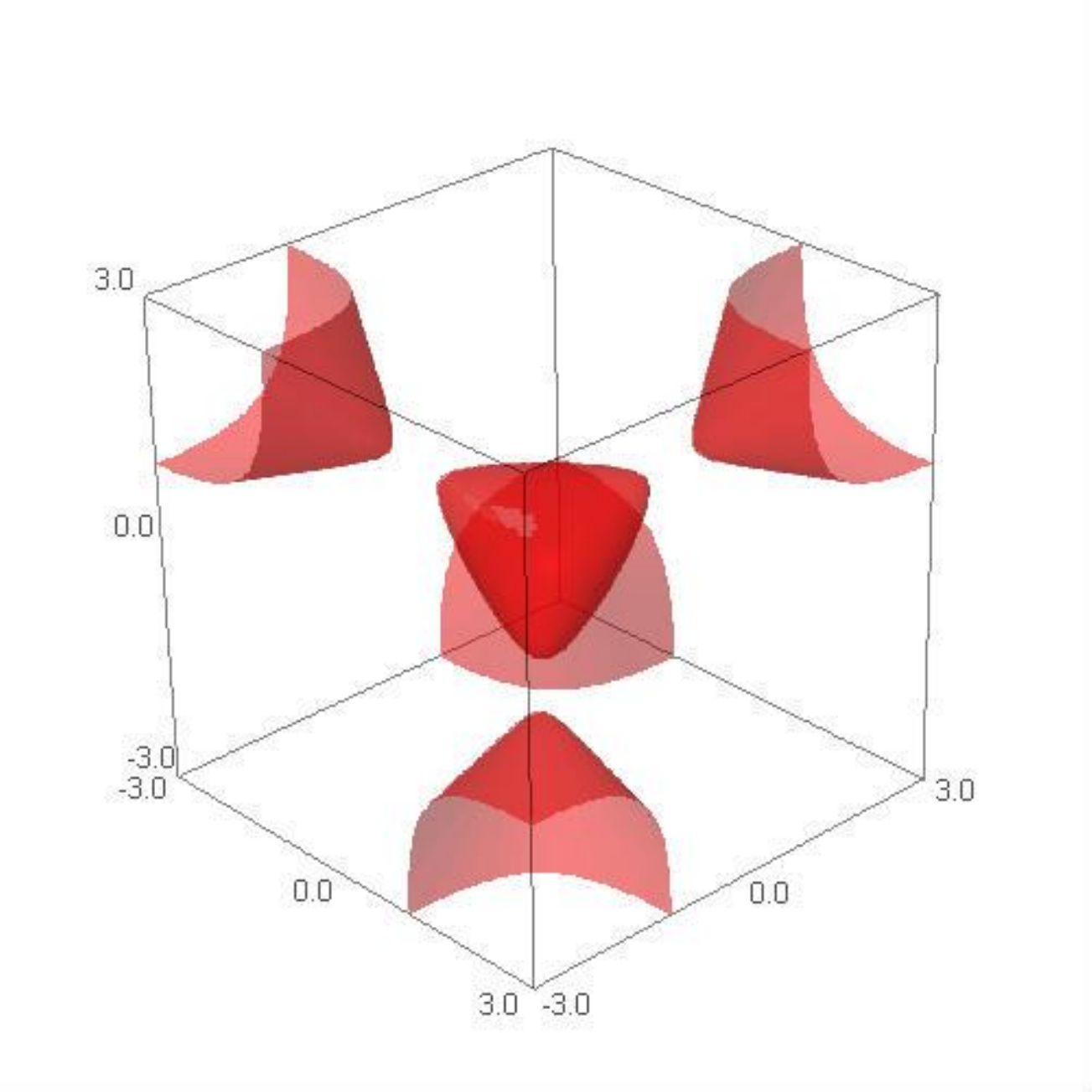}\quad
   \includegraphics[scale=0.3]{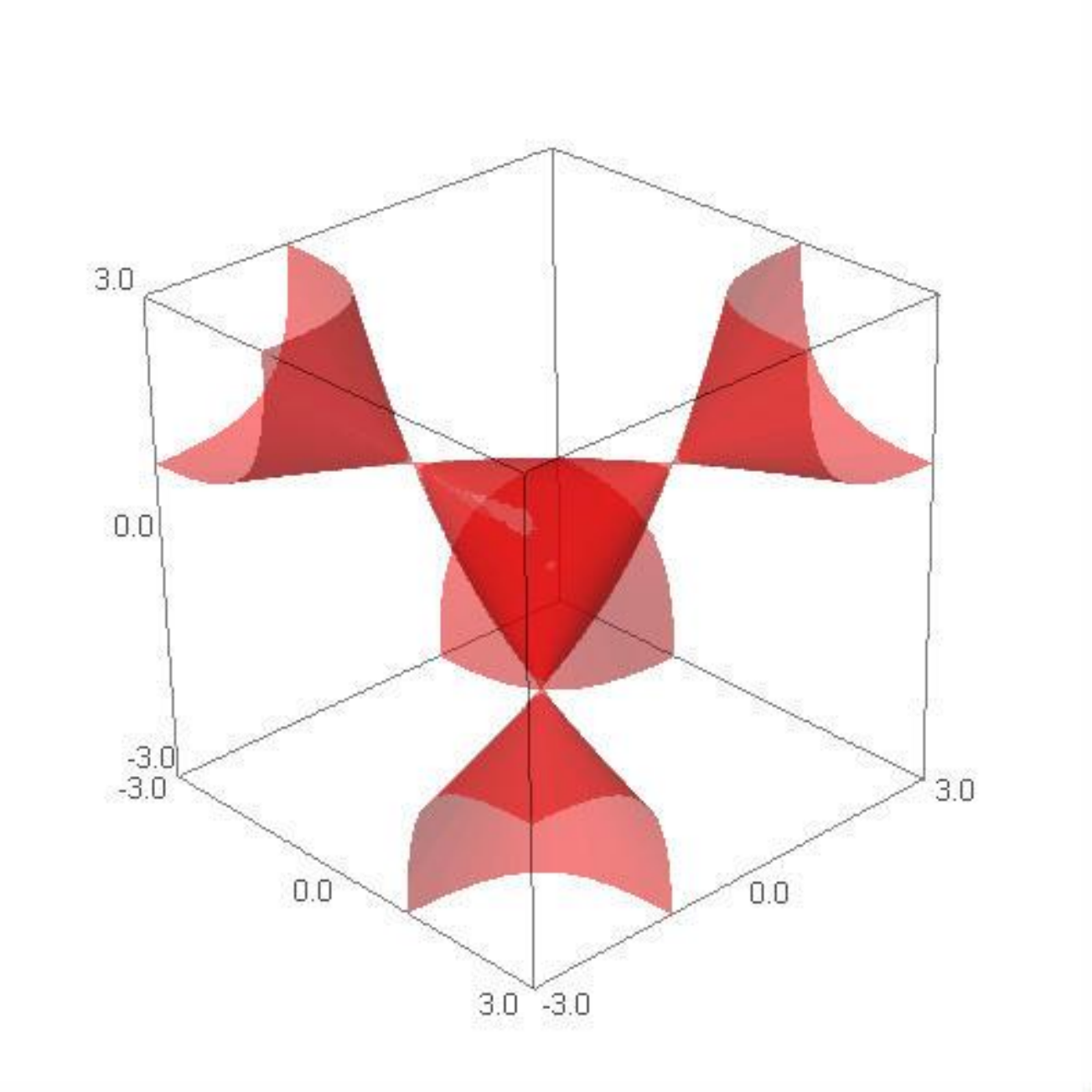}\quad
   \includegraphics[scale=0.3]{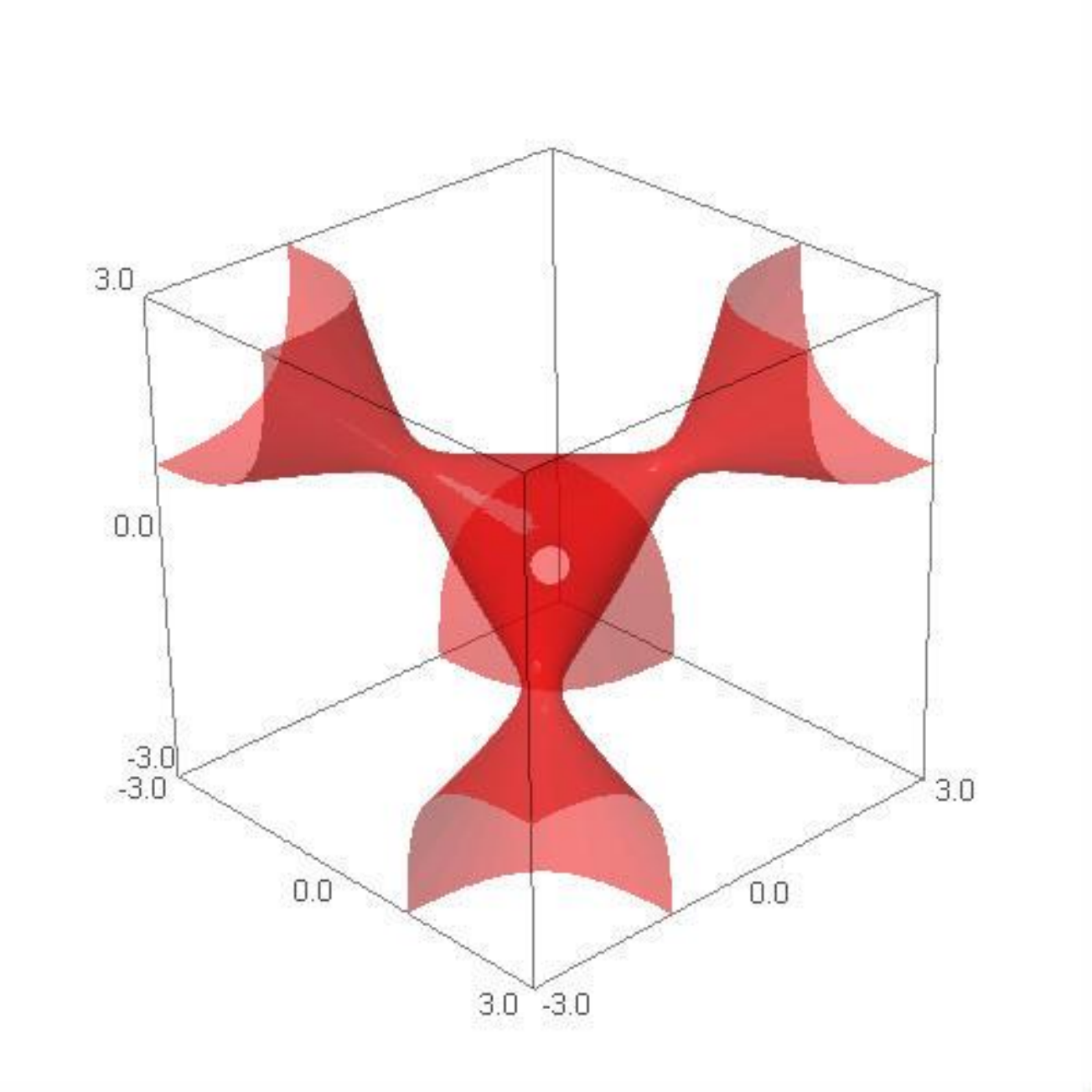}\quad

\end{center}
\vspace*{-1.5em}
\caption{\label{fig:sig}
The surface $S_I$ for $I < 0$, $I = 0$, and $I > 0$.
}
\end{figure}
\bigskip

We will need the following statement about the dynamics of $T$ on $S_I$ for $I < 0$.

\begin{lemma}\label{l.ilessthanzero}
Suppose $I < 0$. Then, every point on the bounded connected component of $S_I$ has a bounded $T$-orbit, and every point on one of the unbounded connected components of $S_I$ has unbounded forward and backward orbit under $T$.
\end{lemma}

\begin{proof}
By \cite[Proposition~4]{BR}, the bounded connected component of $S_I$ is forward and backward invariant under $T$ and hence the $T$-orbit of any point on it is bounded.

On the other hand, by \cite[Theorem~4.3]{R96}, any point on one of the unbounded components of $S_I$ has unbounded forward and backward orbits under $T$.
\end{proof}

The dynamical spectrum $B$ is defined by
$$
B = \{ E \in \R : \{ T^n(x_1(E), x_0(E), x_{-1}(E)) \}_{n \in \Z_+} \text{ is bounded} \}.
$$

\begin{prop}\label{p.sigmab}
We have $\Sigma = B$.
\end{prop}

\begin{proof}
Recall that the one-sided infinite Fibonacci word $u$ is the limit of the finite words $S^n(a)$, $n \ge 0$. Moreover, it is not hard to see that the Fibonacci subshift $\Omega$ contains an element $\omega_u$ whose restriction to the right half-line coincides with $u$. Since the spectrum of $H_\omega$ is $\omega$-independent by minimality, we can study the set $\Sigma$ by considering the spectrum of the particular operator $H_{\omega_u}$. For $E \in \R$ and $n \ge 0$, denote by $M_n(E)$ the transfer matrix corresponding to the operator $H_{\omega_u}$, the energy $E$, and the interval arising from the word $S^n(a)$ through our substitution procedure. Due to $S^{n+1}(a) = S^n(ab) = S^n(a) S^{n-1}(a)$, we have
\begin{equation}\label{e.matrrec}
M_{n+1}(E) = M_{n-1}(E) M_n(E), \quad n \ge 1.
\end{equation}
Using the Cayley-Hamilton theorem, one sees that for $x_n(E) := \frac12 \mathrm{Tr} \, M_n(E)$, $n \ge 0$, we have
\begin{equation}\label{e.tracerec}
x_{n+1}(E) = 2x_n(E) x_{n-1}(E) - x_{n-2}(E), \quad n \ge 2.
\end{equation}
Note that \eqref{e.matrrec} and \eqref{e.tracerec} can be used to \emph{define} $M_n(E)$ and $x_n(E)$ for $n < 0$. In particular,
$$
M_{-1}(E) = M_1(E) M_0(E)^{-1}.
$$
Since $M_1(E)$ corresponds to $S^1(a) = ab$ and $M_0(E)$ corresponds to $S^0(a) = a$, this shows that the matrix $M_{-1}(E)$ is just the transfer matrix across $f_b$ corresponding to the energy $E$. In other words, our current definition of $x_n(E)$ for $n = -1,0,1$ agrees with the definition given above; compare \eqref{e.x-1}--\eqref{e.x1}.

In particular, we have
$$
B = \{ E \in \R : \{ x_n(E) \}_{n \in \Z_+} \text{ is bounded} \}.
$$

We can now prove the inclusion ``$B \subseteq \Sigma$.'' Suppose $E \in B$. As is well known (and in fact easy to check), $u$ does not only have $S^n(a)$, $n \ge 0$ as prefixes, but also $S^n(a) S^n(a)$, $n \ge 2$. Combining this with the boundedness of $x_n(E)$ for $E \in B$ and the Gordon lemma (see, e.g., \cite{D00, DS00, G76}), we find that no solution of $-u'' + V_{\omega_u}u = Eu$ is square-integrable at $+\infty$. This implies by standard arguments that $E \in \sigma(H_{\omega_u}) = \Sigma$.

\bigskip

Next, we prove the inclusion ``$\Sigma \subseteq B$.'' Denote
$$
\sigma_n = \{ E \in \R : |x_n(E)| \le 1 \}.
$$
By Floquet-Bloch theory, $\sigma_n$ is the spectrum of the Schr\"odinger operator $H_n$ with periodic potential obtained by repeating the piece corresponding to $S^n(a)$. The operators $H_n$ converge in strong resolvent sense to $H_\omega$ for some suitably chosen $\omega \in \Omega$. This implies
\begin{equation}\label{e.specincl}
\Sigma = \sigma(H_\omega) \subseteq \bigcap_{n \ge 0} \overline{\bigcup_{k \ge n} \sigma_k}.
\end{equation}
On the other hand, a minor modification of \cite[Proposition~12.8.6]{OPUC2} shows that $E \not\in B$ if and only if there exists $n$ with $|x_{n+1}(E)| > 1$, $|x_{n}(E)| > 1$, and $|x_{n+1}(E) x_{n}(E)| > |x_{n-1}(E)|$. Moreover, in this case, we necessarily have $|x_k(E)| > 1$ for all $k \ge n$.

Combining these two facts, we can now conclude the proof. Suppose $E \not\in B$. Then there exists $n$ with $|x_{n+1}(E)| > 1$, $|x_{n}(E)| > 1$, and $|x_{n+1}(E) x_{n}(E)| > |x_{n-1}(E)|$. Choose $\varepsilon > 0$ such that $|x_{n+1}(E')| > 1$, $|x_{n}(E')| > 1$, and $|x_{n+1}(E') x_{n}(E')| > |x_{n-1}(E')|$ for every $E'$ with $|E - E'| < \varepsilon$. For each of these $E'$, we therefore have $|x_k(E')| > 1$ for all $k \ge n$. This shows that
$$
E \not\in \overline{\bigcup_{k \ge n} \sigma_k}
$$
and hence $E \not\in \Sigma$ by \eqref{e.specincl}.
\end{proof}

At this point we would like to point out a major difference between the discrete case and the continuum case. In the discrete case, $I(E)$ is actually constant, while in the continuum case, it is in general not constant (as we will explicitly see below). As we will see later, this leads to new phenomena in the continuum case that make its study worthwhile.

Moreover, in the discrete case, the invariant is always non-negative. In the continuum case, we cannot a priori rule out that it may be negative for some energies. The following proposition shows that even if this happens, these energies are not essential as they must lie outside the spectrum. As a consequence, the study of the dynamics of the trace map on $S_I$ for $I \ge 0$, which has been investigated heavily in the papers on the discrete Fibonacci Hamiltonian, is sufficient to describe the corresponding spectral properties of the continuum Fibonacci Hamiltonian.

\begin{prop}\label{p.nonneginv}
We have $I(E) \ge 0$ for every $E \in \Sigma$.
\end{prop}

\begin{proof}
Assume there is $E \in \Sigma$ with $I(E) < 0$. Then, since the forward trace map orbit of $(x_2(E) , x_1(E) , x_0(E))$ remains bounded by Proposition~\ref{p.sigmab}, this point must belong to the compact component of the invariant surface $S_{I(E)}$ by Lemma~\ref{l.ilessthanzero}. Now wiggle $E$. By continuity of $x_n(\cdot)$, nearby $E'$'s must have $I(E') < 0$ and $(x_2(E') , x_1(E') , x_0(E'))$ contained in the bounded component of $S_{I(E)}$ as well. Thus, again by Lemma~\ref{l.ilessthanzero}, $(x_1(E') , x_2(E') , x_3(E'))$ has bounded forward trace map orbit, too, and hence $E' \in \Sigma$ by Proposition~\ref{p.sigmab}. It follows that $\Sigma$ is not a Cantor set, which contradicts Theorem~\ref{t.fibcantor}.
\end{proof}

We can now address the local fractal dimension of the spectrum.

\begin{theorem}\label{t.fiblocaldim}
There is a continuous map $D : [0,\infty) \to (0,1]$ with the following properties:
\begin{itemize}

\item[{\rm (i)}] $\dim_\mathrm{H}^\mathrm{loc}(E,\Sigma) = D(I(E))$ for every $E \in \Sigma$.

\item[{\rm (ii)}] We have $D(0) = 1$ and $1 - D(I) \asymp \sqrt{I}$ as $I \downarrow 0$.

\item[{\rm (iii)}] We have
$$
\lim_{I \to \infty} D(I) \cdot \log I = 2 \log (1 + \sqrt{2})
$$

\item[{\rm (iv)}] $D$ is real analytic in $(0, \infty)$.

\end{itemize}
\end{theorem}

\begin{proof}
Recall that by Lemma~\ref{isolated}, $\Sigma$ contains no isolated points. Therefore, by Proposition~\ref{p.nonneginv} and \cite[Theorem 2.13]{DMY}, $\dim_\mathrm{H}^\mathrm{loc} (E, \Sigma)$ for $E \in \Sigma$ depends only on the value of $I(E)$. This implies (i). Given this, (ii) follows from \cite{DG11} and (iii) follows from \cite{DEGT}. (iv) follows from \cite[Theorem~5.23]{Can}. Together with (ii), this implies continuity of $D$ on $[0,\infty)$.
\end{proof}

\subsection{Explicit Computations and Formulae in Special Cases}

In this subsection we consider special choices of the local potential pieces $f_a, f_b$ for which it is possible to compute the Fricke-Vogt invariant $I(E)$ explicitly. As we have seen above, this has a direct connection to the local structure of the spectrum near an energy $E \in \Sigma$. The explicit formulae also permit us to identify the limit behavior of the local Hausdorff dimension of the spectrum in the regime of large energy or small/large coupling.

\subsubsection{A Piecewise Constant Potential}

For comparison purposes, let us start with the free case, $f_a = f_b \equiv 0$. That is, we consider the free Schr\"odinger operator $-\frac{d^2}{dx^2}$ in $L^2(\R)$, but view its (zero) potential as tiled according to a Fibonacci sequence. In this degenerate case we of course do not have aperiodicity and irreducibility. However, all quantities and formulae that arise from the presence of Fibonacci symmetries still exist. In particular, we may compute the curve of initial conditions and the resulting Fricke-Vogt invariant $I(E)$.

For $E > 0$, we have
$$
M(a,E) = M(b,E) = \begin{pmatrix} \cos \sqrt{E} &  \frac{1}{\sqrt{E}}\sin \sqrt{E}\\ -\sqrt{E}\sin\sqrt{E} & \cos \sqrt{E} \end{pmatrix},
$$
and hence $x_{-1}(E) = x_0(E) = \cos \sqrt{E}$. Similarly we obtain $x_1(E) = \cos(2 \sqrt{E})$.

For $E < 0$, we have
$$
M(a,E) = M(b,E) = \begin{pmatrix} \cosh \sqrt{-E} &  \frac{1}{\sqrt{-E}}\sinh \sqrt{-E} \\ \sqrt{-E}\sinh\sqrt{-E}  & \cosh \sqrt{-E} \end{pmatrix},
$$
and hence $x_{-1}(E) = x_0(E) = \cosh \sqrt{-E}$, $x_1(E) = \cosh (2 \sqrt{-E})$.

A direct calculation shows now that in both cases $I(E) = 0$. By continuity, we have also $I(0) = 0$. Thus, as in the discrete setting, the invariant vanishes identically in the free case. This suggests that the dynamical behavior of the trace map near $S_0$ is essential for a study of the continuum Fibonacci Hamiltonian in the small coupling regime (at least in compact energy regions), just as it was the case in the discrete setting and which has led to numerous recent advances \cite{DG09, DG11, DG12, DG13}.

\bigskip

Consider now the case $f_a = \lambda \cdot \chi_{[0,1)}$ and $f_b = 0 \cdot \chi_{[0,1)}$, where $\lambda \ge 0$. Clearly, when $\lambda = 0$, we obtain the free case considered above; and when $\lambda > 0$, the resulting potentials are aperiodic and irreducible. Let us compute the initial conditions $(x_1(E) , x_0(E) , x_{-1}(E))$ first in the case $E > \lambda$. We have (compare with (2.15) from \cite{TK}):
\begin{align*}
x_{-1}(E) & = \cos \sqrt{E}, \\
x_0(E) & = \cos \sqrt{E - \lambda}, \\
x_1(E) & = \cos \sqrt{E} \cos \sqrt{E-\lambda} -\frac{1}{2}\left(\sqrt{\frac{E}{E-\lambda}}+\sqrt{\frac{E-\lambda}{E}}\right)\sin \sqrt{E} \sin \sqrt{E-\lambda}.
\end{align*}
From this we can calculate $I(E)$ explicitly:
\begin{equation}\label{e.1}
I(E) = \frac{1}{4}\frac{\lambda^2}{E(E-\lambda)}\sin^2 \sqrt{E} \sin^2 \sqrt{E-\lambda}.
\end{equation}
Notice that if $\sin \sqrt{E} = 0$, then $\cos \sqrt{E} \in \{-1,1\}$, and we have either $(x_1(E) , x_0(E) , x_{-1}(E)) = (\cos \sqrt{E-\lambda}, \cos \sqrt{E-\lambda}, 1)$, which is close to $(1, 1, 1)$ if $E \gg 1$, or $(x_1(E) , x_0(E) , x_{-1}(E)) = (-\cos \sqrt{E-\lambda}, \cos \sqrt{E-\lambda}, -1)$, which is close to $(1, -1, -1)$ if $E \gg 1$.

Let us now calculate the initial conditions and $I(E)$ for $E \in (0, \lambda)$. We have (compare with (2.14) from \cite{TK}):
\begin{align*}
x_{-1}(E) & = \cos \sqrt{E}, \\
x_0(E) & = \cosh \sqrt{E - \lambda}, \\
x_1(E) & = \cos \sqrt{E} \cosh \sqrt{\lambda-E}+\frac{1}{2}\left(\sqrt{\frac{\lambda-E}{E}}-\sqrt{\frac{E}{E-\lambda}}\right)\sin \sqrt{E} \sinh \sqrt{\lambda-E},
\end{align*}
and hence (compare with (3.4) from \cite{TK})
\begin{equation}\label{e.2}
I(E) = \frac{1}{4} \frac{\lambda^2}{E(\lambda - E)} \sin^2 \sqrt{E} \sinh^2 \sqrt{\lambda-E}.
\end{equation}

With these explicit formulae, we can now establish the asymptotics of the local Hausdorff dimension of the spectrum in the small coupling regime as well as in the high energy regime.

\begin{coro}\label{c.1}
We have the following asymptotics:
\begin{align*}
\lim_{\lambda \to 0} \; \inf_{E \in \Sigma} \; \dim_\mathrm{H}^\mathrm{loc}(E,\Sigma) & = 1, \\
\lim_{K \to \infty} \; \inf_{E \in \Sigma \cap [K,\infty)} \; \dim_\mathrm{H}^\mathrm{loc}(E,\Sigma) & = 1.
\end{align*}
\end{coro}

\begin{proof}
From \eqref{e.1} and \eqref{e.2} we know that $I(E) \to 0$ as $\lambda \to 0$, uniformly in $E \ge 0$ (hence uniformly in $E \in \Sigma$). This implies the first statement of Corollary \ref{c.1} due to Theorem~\ref{t.fiblocaldim}.

Also, from \eqref{e.1} we see that for any fixed value of $\lambda$, we have $I(E) \to 0$ as $E \to \infty$, which implies the second statement, again by Theorem~\ref{t.fiblocaldim}.
\end{proof}

In the large coupling regime, we have the following result:

\begin{coro}\label{c.2}
For any compact $S \subset \R$, we have
$$
\lim_{\lambda \to \infty} \dim_\mathrm{H}(\Sigma \cap S) = 0.
$$
\end{coro}

\begin{proof}
Consider $E \in S$ with $I(E)\le N^2$, and so that $\lambda \gg E, N$. Then, due to \eqref{e.2},
$$
\sqrt{I(E)} \sim \lambda^{1/2} \left|\sin \sqrt{E} \sinh \sqrt{\lambda - E} \right| \le N,
$$
and hence $|\sin \sqrt{E}| \lesssim \frac{N}{\lambda^{1/2} e^{\lambda^{1/2}}} \ll 1$. This implies that $|\cos \sqrt{E}| \ge 1/2$. Now we have
$$
|x_{-1}(E)| = |\cos \sqrt{E}| \le 1, |x_0(E)| = |\cosh \sqrt{\lambda - E}| \sim e^{\lambda^{1/2}} > 1,
$$
and
\begin{align*}
|x_1(E)| & = \left| \cos \sqrt{E} \cosh \sqrt{\lambda - E} + \frac{1}{2} \left( \sqrt{\frac{\lambda - E}{E}} - \sqrt{\frac{E}{E - \lambda}} \right) \sin \sqrt{E} \sinh \sqrt{\lambda - E}\right| \\
& \gtrsim e^{\lambda^{1/2}} - C \lambda^{1/2} \left| \sin \sqrt{E} \sinh \sqrt{\lambda - E} \right| \\
& \gtrsim e^{\lambda^{1/2}} - N \\
& > 1.
\end{align*}
This implies $E \not\in B$ by \cite[Lemma~2]{S87} and hence $E \not \in \Sigma$ by Proposition~\ref{p.sigmab}. Therefore, for any $E \in S \cap \Sigma$, we must have $I(E) > N^2$ if $\lambda$ is sufficiently large, and this proves Corollary~\ref{c.2}.
\end{proof}

One in general expects that in the regime of small coupling or high energy, the local characteristics of the spectrum should approach those of the free case. Similarly, one expects that for fixed energy (or fixed compact energy region), the potential should dominate in the large coupling regime and hence the situation should become as singular as possible. Corollaries~\ref{c.1} and \ref{c.2} show that the expected behavior holds true in the simple situation of a piecewise constant potential, where the Fricke-Vogt invariant can be computed explicitly. It is quite natural to expect that these results should extend to much more general choices of the potential pieces. In fact, we ask the following:

\begin{quest}
Is it true that Corollaries~\ref{c.1} and \ref{c.2} hold regardless of the shape of the bump? That is, if we replace $f_a = \lambda \cdot \chi_{[0,1)}$ and $f_b = 0 \cdot \chi_{[0,1)}$ by general $f_a \in L^2(0,\ell_a)$ and $f_b \in L^2(0,\ell_b)$, do Corollaries~\ref{c.1} and \ref{c.2} continue to hold as stated?
\end{quest}

\subsubsection{The Kronig-Penney Model}

The Kronig-Penney model places local point interactions at a discrete set of points. Here we are interested in the case where the location of these points is dictated by a Fibonacci sequence. This model was considered in \cite{BJK} in an ``off-diagonal'' setting. Let us make the calculations also in the following (``diagonal'') setting: We take $\ell_a = \ell_b = 1$ and  $f_a(x) = \lambda \delta(x)$, $f_b(x) = 0$. Note that strictly speaking, this model is not contained in the general class of models considered earlier in the paper. We nevertheless compute the invariant in this case as, together with our earlier results, this is quite constructive. We leave the extension of the general part to the interested reader (note that \cite{KLS} actually did consider measures rather than potentials).

For $E > 0$, we have
\begin{align*}
M(b,E) & = \begin{pmatrix} \cos \sqrt{E} &  \frac{1}{\sqrt{E}}\sin \sqrt{E}\\ -\sqrt{E}\sin\sqrt{E} & \cos \sqrt{E} \end{pmatrix}, \\
M(a,E) & = \begin{pmatrix} \cos \sqrt{E} &  \frac{1}{\sqrt{E}}\sin \sqrt{E}\\ -\sqrt{E}\sin\sqrt{E} & \cos \sqrt{E} \end{pmatrix}\begin{pmatrix} 1 &  0\\ \lambda & 1 \end{pmatrix} \\
& = \begin{pmatrix} \cos \sqrt{E}+\frac{\lambda}{\sqrt{E}}\sin \sqrt{E} &  \frac{1}{\sqrt{E}}\sin \sqrt{E}\\ \lambda\cos \sqrt{E}-\sqrt{E}\sin\sqrt{E} & \cos \sqrt{E} \end{pmatrix}.
\end{align*}

In this case we have:
$$
\left\{
  \begin{array}{ll}
    x_{-1}(E)=\cos\sqrt{E},  \\
    x_0(E)=\cos2\sqrt{E}+\frac{\lambda}{2\sqrt{E}}\sin \sqrt{E},  \\
    x_1(E)=\cos2\sqrt{E}+\frac{\lambda}{2\sqrt{E}}\sin 2\sqrt{E}.
  \end{array}
\right.
$$

Explicit calculations show that in this case
\begin{equation}\label{e.kpinv}
I(E) = \frac{\lambda^2}{4 E} \sin^2 \sqrt{E}
\end{equation}
(compare with Section 2 in \cite{BJK}). Notice that, irrespectively of the value of the coupling constant $\lambda$, if $\sin \sqrt{E} = 0$ for some energy, then the formulae above show that this energy is in the spectrum, and the corresponding point is on internal part of the Cayley cubic. Since there are no isolated points in the spectrum, this implies that at those points the local Hausdorff dimension of the spectrum is equal to one (which explains the nature of ``pseudo bands'' mentioned in \cite{BJK}, see also \cite{KS} where a similar observation was made), and in all other points of the spectrum, its local Hausdorff dimension is strictly less than one. In particular this shows that the analog of Corollary \ref{c.2} does not hold in this case, which shows an essential difference between the two simple models we have considered.

Certainly, due to \eqref{e.kpinv}, the analog of Corollary \ref{c.1} does hold for this model. This shows that at least the heuristics leading to the general expectation in the small coupling and high energy regimes do apply in this case.

\section*{Acknowledgments} The plots of $S_I$ for $I < 0$, $I = 0$, and $I > 0$ we show in Figure~1 were generated by William Yessen. We are grateful to him for the permission to include them here. We also thank Christian Remling and G\"unter Stolz for useful comments on the literature.


\end{document}